\documentclass[11pt,reqno]{amsart}
\usepackage{amsmath}
\usepackage{amssymb}
\usepackage{graphicx}
\newtheorem{theorem}{Theorem}[section]
\newtheorem{prop}[theorem]{Proposition}
\newtheorem{lemma}[theorem]{Lemma}
\newtheorem{cor}[theorem]{Corollary}

\theoremstyle{definition}

\newtheorem{defn}[theorem]{Definition}

\newtheorem{example}[theorem]{Example}
\newtheorem{remark}[theorem]{Remark}
\newtheorem{notation}[theorem]{Notation}
\newtheorem{conjecture}[theorem]{Conjecture}
\newtheorem{question}[theorem]{Question}

\def\X{\mathcal X}

\def\m{\mathcal M}

\def\R{\mathbb R}

\def\phi{\varphi}
\def\rho{\varrho}
\def\epsilon{\varepsilon}

\begin{document}

\newcounter{algnum}
\newcounter{step}
\newtheorem{alg}{Algorithm}

\newenvironment{algorithm}{\begin{alg}\end{alg}}

\title{How many inflections are there in the Lyapunov spectrum?}

\author{O. Jenkinson, M. Pollicott \& P. Vytnova}

\begin{abstract}
Iommi \& Kiwi \cite{iommikiwi} showed that the Lyapunov spectrum of an expanding map need not be concave,
and posed various problems concerning the possible number of inflection points.
In this paper we answer a conjecture in \cite{iommikiwi} by proving that the Lyapunov spectrum of a two branch piecewise linear map has at most two points of inflection.
We then answer a question in  \cite{iommikiwi} by proving that there exist finite branch piecewise linear maps whose Lyapunov spectra have arbitrarily many points of inflection.
This approach is used to 
exhibit a countable branch piecewise linear map whose Lyapunov spectrum has infinitely many points of inflection.
\end{abstract}

\thanks{The authors would like to thank Victor Klepsyn for his very helpful suggestions.
The second author was partly supported by the ERC Grant
833802-Resonances.  The third author was partly supported by EPSRC grant EP/T001674/1.
}
\maketitle

\section{Introduction}\label{introsection}

For a differentiable dynamical system $T:X\to X$, 
where for simplicity $X$ is a subset of the unit interval,
the \emph{Lyapunov exponent}
of a point $x\in X$ is given by
$$\lambda(x)=\lim_{n\to\infty} \frac{1}{n}\log |(T^n)'(x)|
$$
whenever this limit exists.
Typically the set of all Lyapunov exponents for a given map $T$ is a closed interval of positive length. 
An investigation into the size of the set of points $x$ corresponding to a given Lyapunov exponent
$\alpha$ in this interval leads to the notion,
introduced by Eckmann \& Procaccia \cite{eckmannprocaccia},
 of the associated \emph{Lyapunov spectrum} $L$,
being a map given by defining $L(\alpha)$ as the Hausdorff dimension of the level set $\{x\in X: \lambda(x)=\alpha\}$.

The Lyapunov spectrum was studied rigorously by Weiss \cite{weiss}, continuing a broader programme with Pesin
(see e.g.~\cite{pesinbook, pesinweiss}).  
In the setting of conformal expanding maps with finitely many branches, 
Weiss \cite{weiss} proved the real analyticity of $L$,
and also claimed that $L$ is always concave\footnote{In fact the word \emph{convex} (rather than \emph{concave}) is used in the claim \cite[Thm.~2.4 (1)]{weiss}, though 
 the interpretation is that of  \emph{concave}  in the sense that we use it  (see \cite[p.~536]{iommikiwi}); all specific examples of Lyapunov spectra known at the time of \cite{weiss} were indeed concave.}.
By contrast for expanding maps (on a subset of $[0,1]$, say) with \emph{infinitely} many branches
the Lyapunov spectrum $L$ can never be concave (see e.g.~\cite[Thm.~4.3]{iommi}), a simple consequence of the non-negativity of $L$ and the unboundedness of its domain (i.e.~the interval of all Lyapunov exponents);
the Lyapunov spectrum in the specific case of the Gauss map has been analysed by  Kesseb\"ohmer \& Stratmann \cite{ks},
and in the case of the R\'enyi map by Iommi \cite{iommi}.

Motivated by these examples, Iommi \& Kiwi \cite{iommikiwi} revisited the case of finite branch expanding maps, and discovered that in fact the Lyapunov spectrum is \emph{not} always concave;
indeed 
even in the simplest possible setting of \emph{two-branch piecewise linear maps} 
(see Definition \ref{piecewiselinearmapdefn})
there exist examples with non-concave Lyapunov spectra (so such examples have points of inflection, i.e.~points at which the second derivative vanishes). 
For finite branch maps the number of inflection points is necessarily even (cf.~\cite[p.539]{iommikiwi}), and all examples of non-concave Lyapunov spectra exhibited in \cite{iommikiwi} have precisely two points of inflection.

The natural problem suggested by the work of Iommi \& Kiwi is the extent to which it is possible to find Lyapunov spectra with strictly more than two points of inflection.
Specifically, the following conjecture and question are contained in \cite[p.539]{iommikiwi}:

\begin{conjecture}\label{ikconjecture} (Iommi \& Kiwi \cite{iommikiwi})
The Lyapunov spectrum of a 2-branch expanding map has at most two points of inflection.
\end{conjecture}

\begin{question}\label{ikquestion} (Iommi \& Kiwi \cite{iommikiwi})
Is there an upper bound on the number of inflection points of the Lyapunov spectrum for
piecewise expanding maps?
\end{question}

More broadly, the work of Iommi \& Kiwi provokes interest in constructing maps whose Lyapunov spectra have more than two inflection points, and in understanding the general properties responsible for producing such inflection points.
Henceforth for brevity we shall often use the term \emph{Lyapunov inflection} (of a map) to denote a point of inflection in the Lyapunov spectrum of that map.

In this article we address both Conjecture \ref{ikconjecture} and Question \ref{ikquestion}, as well as the more general issue of understanding maps with more than two Lyapunov inflections.
Specifically, we first give an affirmative answer to Conjecture \ref{ikconjecture} in the setting of 2-branch \emph{piecewise linear} maps:

\begin{theorem}\label{twotheorem}
The Lyapunov spectrum of a two branch piecewise linear map has at most two points of inflection.
\end{theorem} 

We do not know, however, whether there are nonlinear 2-branch expanding maps with more than two Lyapunov inflections.

Secondly, we construct explicit examples of 
maps with more than two Lyapunov inflections
(see e.g.~Figures \ref{4figure} and \ref{6figure} below), and
resolve Question \ref{ikquestion} as follows, proving that there is \emph{no upper bound} on the number of Lyapunov inflections, and that indeed this can be established within the class of piecewise linear maps:

\begin{theorem}\label{arbitrarilymanytheorem}
For any integer $n\ge 0$, there is a piecewise linear map whose Lyapunov spectrum has at least $n$
points of inflection.
\end{theorem}

\begin{figure}[!h]
\begin{center}
\begin{tabular}{cc}
\includegraphics[scale=1.0]{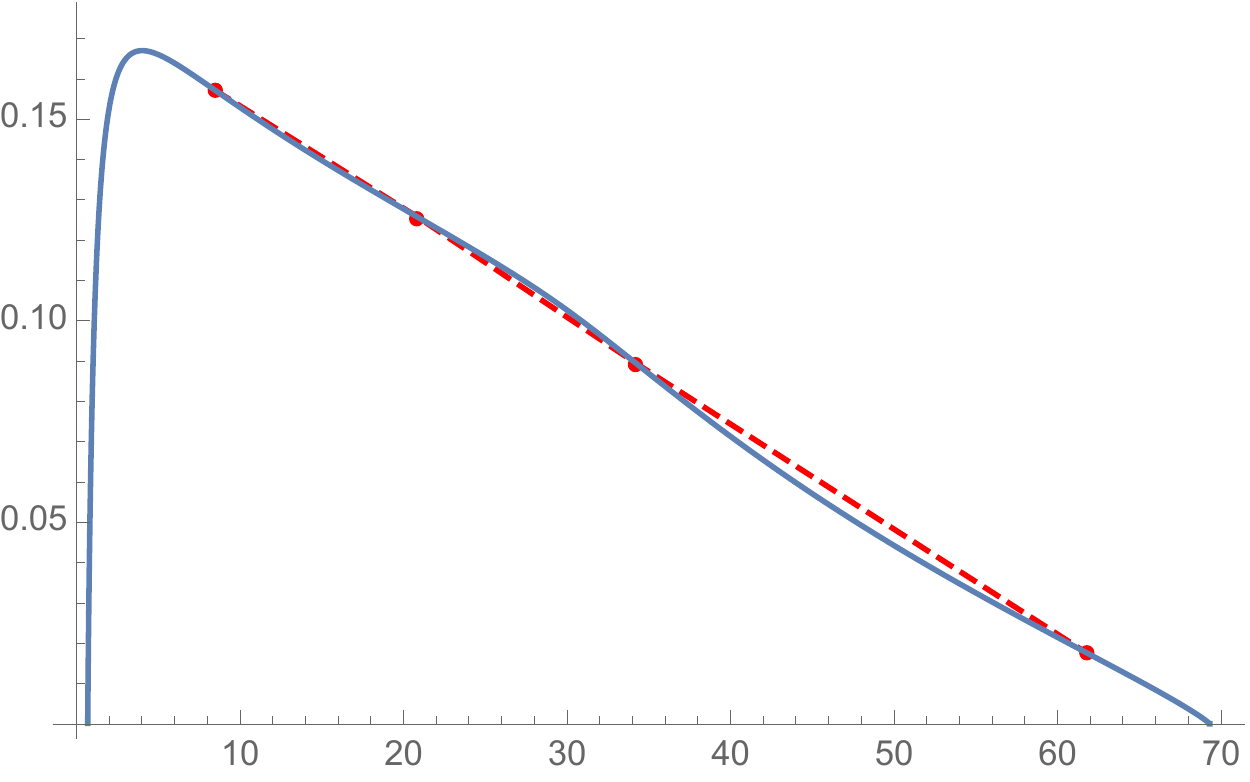}
\end{tabular}
\caption{Graph of Lyapunov spectrum $L$ for 22-branch piecewise linear map $T$ with derivative $T'\equiv2$ on one branch,
 $T'\equiv 2^{45}$ on 20 branches, and $T'\equiv 2^{100}$ on one branch.
 The four Lyapunov inflections 
 (with dashed linear interpolations between them)
 are at 
 $\alpha_1\approx8.52$, $\alpha_2\approx20.88$, $\alpha_3\approx34.22$,
 $\alpha_4\approx61.73$, with $L$ convex on $[\alpha_1,\alpha_2]$ and $[\alpha_3,\alpha_4]$, and concave on complementary intervals in its domain.
 }\label{4figure}
\end{center}
\end{figure}

\begin{figure}[!h]
\begin{center}
\begin{tabular}{cc}
\includegraphics[scale=1.0]{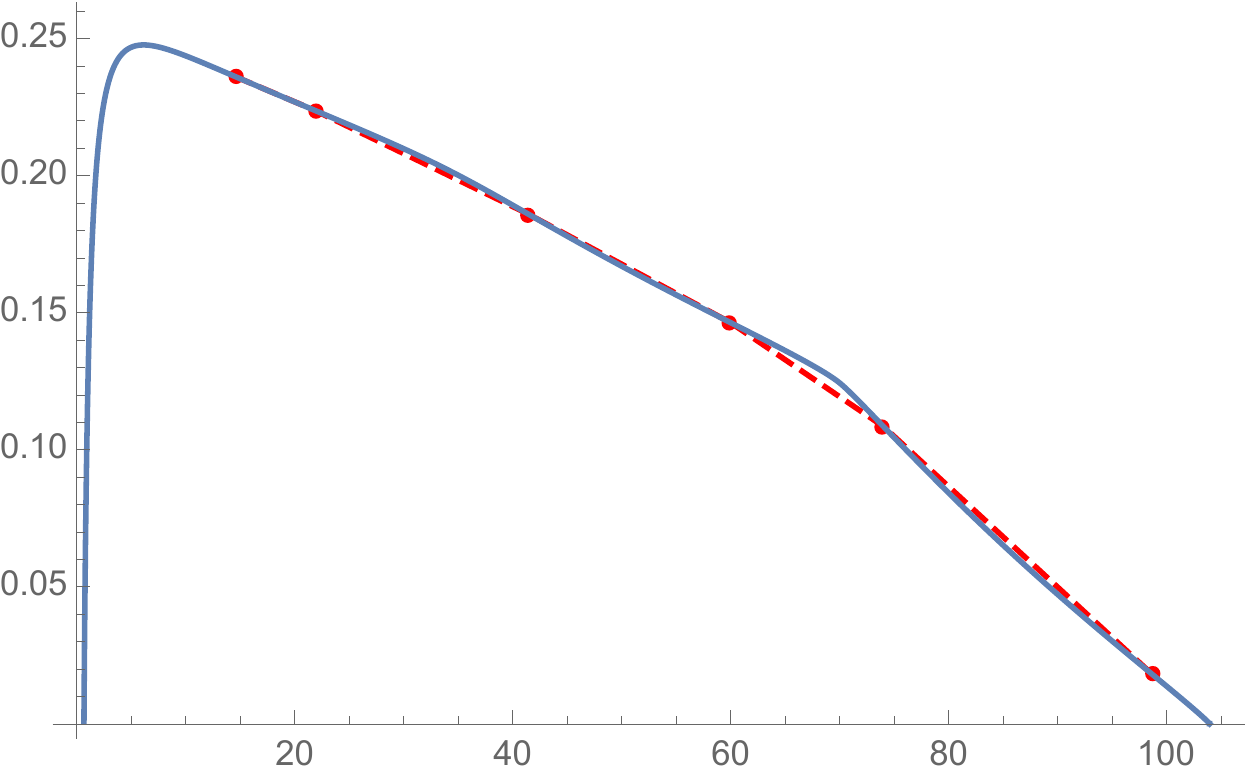}
\end{tabular}
\caption{Graph of Lyapunov spectrum $L$ for 7002-branch piecewise linear map $T$ with derivative $T'\equiv2$ on one branch,
 $T'\equiv 2^{51}$ on 1000 branches, 
 $T'\equiv 2^{101}$ on 6000 branches,
 and $T'\equiv 2^{150}$ on one branch.
 The six Lyapunov inflections 
 (with dashed linear interpolations between them)
 are at 
 $\alpha_1\approx14.66$, $\alpha_2\approx21.85$, $\alpha_3\approx41.48$,
 $\alpha_4\approx60.01$, 
  $\alpha_5\approx74.04$,
   $\alpha_6\approx98.65$,
 with $L$ convex on $[\alpha_1,\alpha_2]$, $[\alpha_3,\alpha_4]$ and $[\alpha_5,\alpha_6]$, and concave on complementary intervals in its domain.}\label{6figure}
\end{center}
\end{figure}

A natural corollary of Theorem \ref{arbitrarilymanytheorem} is that there is also no upper bound on the number of zeros of higher order derivatives of the Lyapunov spectrum:

\begin{cor}\label{arbitrarilymanyhigherordertheorem}
For any integers $n\ge 0$, $k\ge 2$, there is a piecewise linear map whose Lyapunov spectrum has at least $n$
points at which its $k^{th}$ order derivative vanishes.
\end{cor}

A natural by-product of our approach to proving Theorem \ref{arbitrarilymanytheorem} 
is that by moving into the realm of \emph{infinite branch} maps,  the first example of
a map with \emph{infinitely many} Lyapunov inflections can be exhibited:

\begin{theorem}\label{infiniteintrotheorem}
There is an infinite branch piecewise linear map whose Lyapunov spectrum
has a countable infinity of inflection points.
\end{theorem}

As above, this implies a corresponding result for zeros of higher order derivatives of the Lyapunov spectrum:

\begin{cor}
There is an infinite branch piecewise linear map such that
for all $k\ge 2$, its Lyapunov spectrum has
its $k^{th}$ order derivative equal to zero at infinitely many distinct points.
\end{cor}

The organisation of this article is as follows.
Section \ref{preliminariessection} consists of various preliminary definitions and results concerning the Lyapunov spectrum and its first two derivatives. Although we work exclusively with piecewise linear maps (see Definition
\ref{piecewiselinearmapdefn}), much of \S \ref{preliminariessection} is valid in the more general setting of expanding maps.
While most of \S \ref{preliminariessection} is already in the literature in some form, 
our subsequent focus on inflection points motivates the careful derivation of the formula for the second derivative of $L$
(see \S \ref{derivativessubsection} and \S \ref{pwlinearsubsection}) in a way 
that is relatively self-contained. The key ingredients here are a characterisation of $L$ due to Feng, Lau \& Wu \cite{flw}
(see Proposition \ref{Lcharacterisation1prop}),
together with the well known formula (\ref{pderivformula}) for the derivative of pressure. 

In \S \ref{twobranchsection} we prove Theorem \ref{twotheorem}, exploiting an explicit formula for the Lyapunov spectrum  in order to show that it has at most two points of inflection as a consequence of a more general result 
(Theorem \ref{generaltheorem}) concerning symmetric functions whose lower order derivatives are of prescribed sign.

In \S \ref{manyinflectionssection} we prove Theorem \ref{arbitrarilymanytheorem}, by exhibiting piecewise linear maps 
together with explicit lower bounds on the number of their Lyapunov inflections.
More precisely, we define a sequence of maps $T_N$, where  the lower bound on the number of Lyapunov inflections
for $T_N$ grows linearly with $N$.
It is possible to view the maps $T_N$ as \emph{evolving} from each other, in the sense that
each map $T_{N+1}$ can be described in terms of adjoining additional branches to those of $T_N$.
At each stage the adjoined branches have derivatives much larger than the existing branches, a phenomenon reminiscent of the construction of Iommi \& Kiwi \cite{iommikiwi}, who showed
that the 2-branch piecewise linear maps with Lyapunov inflections are such that the derivative on one branch is much larger than that on the other branch (in a certain precise sense, see \cite[Thm.~A]{iommikiwi},
and Theorem \ref{2branchquantitativetheorem}).
The strategy for bounding from below the number of Lyapunov inflections
for $T_N$ exploits the characterisation of inflection points as solutions to an explicit equation
(namely (\ref{keyequation_initial}), derived in Proposition \ref{keyequationprop}) involving a function related to 
the pressure of a certain family of potentials.
The $T_N$ are then constructed so as to facilitate the definition of two interlaced sequences of numbers converging to zero, with the property that one side of the equation is dominant along one sequence, and the other side dominant along the other sequence, up to a certain point (increasing with $N$) in the sequences.
Consideration of the intervals defined by consecutive points in the two interlaced sequences then yields at least one solution to (\ref{keyequation_initial})
in each such interval, up to a certain point that grows with $N$, thereby guaranteeing an increasing number of Lyapunov inflections for the maps $T_N$.

In \S \ref{infinitebranchsection} we see that
the coherence of the construction of the $T_N$ 
produces, by allowing the process to evolve indefinitely, an infinite branch piecewise linear map $T$
(alternatively, the map $T$ could be considered as the primary object, with the $T_N$ viewed
as finite branch truncations of $T$).
Minor modifications
to the approach of \S \ref{manyinflectionssection} then yield Theorem \ref{infiniteintrotheorem}.

Lastly, in \S \ref{examplessection}, we present several explicit examples of piecewise linear maps with
a prescribed number of Lyapunov inflections.
While these examples are of a rather ad hoc nature, it is noteworthy that (unlike the $T_N$ of
\S \ref{manyinflectionssection}) it is possible to give an \emph{exact count} for the number of Lyapunov inflections in each case, and moreover
the number of branches needed in order to produce a given number of Lyapunov inflections is 
more economical
than in \S \ref{manyinflectionssection}. 
This prompts a natural question (Question \ref{minbranches}): what is the minimum number of branches needed
in order to witness a given number of Lyapunov inflections?

\section{Preliminaries}\label{preliminariessection}

\subsection{Piecewise linear maps}

We shall be interested in
full branch piecewise affine maps defined on subsets of the unit interval;
for brevity we call such maps \emph{piecewise linear}:

\begin{defn}\label{piecewiselinearmapdefn}
Given an integer $q\ge 2$, let 
$\{X_i\}_{i=1}^q$ be a collection of pairwise disjoint closed sub-intervals of $[0,1]$, with lengths $|X_i|>0$.
An associated \emph{piecewise linear map} is any of the $2^q$ maps
$\cup_{i=1}^q X_i\to[0,1]$ whose restriction to each $X_i$ is an affine homeomorphism onto $[0,1]$
(necessarily with derivative $\pm |X_i|^{-1}$).
Any restriction of the map to an interval $X_i$ is referred to as
a \emph{branch}.

For any piecewise linear map $T$, the set
$X:= \{x\in[0,1]: T^n(x)\in \cup_{i=1}^q X_i \text{ for all } n\ge 0\}$
depends only on the collection $\{X_i\}_{i=1}^q$, and is such that
the restricted piecewise linear map $T:X\to X$ is surjective.
We refer to $X$ as the associated \emph{invariant set}, and henceforth
always consider piecewise linear maps as dynamical systems $T:X\to X$.
\end{defn}

\begin{remark}\label{openingremark}
\item[\, (a)]
The pairwise disjointness of the $X_i$ means the invariant set $X$ is a Cantor set
(and is \emph{self-similar}, cf.~\cite[Ch.~9]{falconer}), whose Hausdorff dimension is the unique value $s$ such that
$\sum_{i=1}^q |X_i|^s =1$ (a result essentially due to Moran \cite{moran}, see also e.g.~\cite[Thm.~9.3]{falconer}).
\item[\, (b)] A minor variant of Definition \ref{piecewiselinearmapdefn} would have been to only insist that the intervals $X_i$ have pairwise disjoint interiors (i.e.~allow possible intersections at their endpoints); this would have involved choosing the value of $T$ at any such points of intersection, but otherwise the theory would have been identical to that developed here.
\item[\, (c)] Since each $X_i$ has length strictly smaller than 1, the derivative
$\pm |X_i|^{-1}$ of the piecewise linear map on $X_i$ is in modulus strictly larger than 1, so in particular the map is
\emph{expanding}.
It should be noted that the discussion in the following subsections \ref{lyap2subsection}, \ref{characterisationssubsection} and \ref{derivativessubsection} in fact applies to more general expanding maps (i.e.~where the restriction to each $X_i$ is a diffeomorphism onto $[0,1]$ with derivative strictly larger than 1 in modulus),
and only later (from \S \ref{pwlinearsubsection} onwards) do we require the piecewise linear assumption.
\item[\, (d)] Our piecewise linear maps were referred to as \emph{linear cookie-cutters} in \cite{iommikiwi}, 
following e.g.~\cite{bedford, bohrrand,sullivan}.
\item[\, (e)] In Section \ref{twobranchsection} we shall be concerned with the general two-branch case, i.e.~$q= 2$.
In Section \ref{manyinflectionssection} we shall deal with particularly large values of $q$ in order to guarantee Lyapunov spectra with many points of inflection. 
\end{remark}

\begin{notation}
Let $\m$ denote the set of $T$-invariant Borel probability measures on $X$.
\end{notation}

\subsection{Lyapunov exponents and the Lyapunov spectrum}\label{lyap2subsection}

\begin{defn}
 The \emph{Lyapunov exponent} $\lambda(x)$ of a point $x \in X$ is
 defined by 
 $$\lambda(x) =
 \lim_{n \to+\infty} \frac{1}{n} \log |(T^n)'(x)|
 $$
 whenever this limit exists,
and 
 for a measure $\mu\in\m$, its \emph{Lyapunov exponent} $\Lambda(\mu)$ is defined by
 $$
 \Lambda(\mu)= \int \log|T'|\, d\mu \,.
 $$
  \end{defn}

\begin{remark}
Naturally there is a relation between the two notions of Lyapunov exponent: if $\mu\in\m$ is ergodic then
 $\lambda(x)=\Lambda(\mu)$ for $\mu$-almost every $x$, by the ergodic theorem,
 since $\log |(T^n)'(x)| = \sum_{i=0}^{n-1} \log |T'(T^ix)|$.
\end{remark}

\begin{defn}
Since $\log|T'|$ is continuous, and $\m$ is both convex and weak-$*$ compact (see e.g.~\cite{walters}), it follows that the set of all possible Lyapunov exponents is a closed interval, which we shall denote by $A = [\alpha_{\min}, \alpha_{\max}]$.
More precisely, if $X'$ denotes the set of those $x\in X$ for which 
$\lambda(x) =
 \lim_{n \to+\infty} \frac{1}{n} \log |(T^n)'(x)|$ exists, then 
the \emph{domain}
$A$ is defined by
 $$
 A=[\alpha_{\min}, \alpha_{\max}] = \Lambda(\m) = \lambda(X') \,.
 $$
 \end{defn}

\begin{remark}
 Note that the endpoints $\alpha_{\min}$ and $\alpha_{\max}$ are, respectively, the minimum and the maximum Lyapunov exponent, and can be characterised as
 $$
 \alpha_{\min} = \min_{\mu\in\m} \Lambda(\mu) = \min_{x\in X'} \lambda(x)
 $$
 and
 $$
 \alpha_{\max} = \max_{\mu\in\m} \Lambda(\mu) = \max_{x\in X'} \lambda(x) \,.
 $$
\end{remark}

\begin{notation}
For $\alpha\in \R$ let us write 
$$
X_\alpha = \lambda^{-1}(\alpha) = \{x\in X': \lambda(x)=\alpha\}  \,,
$$
$$\m_\alpha = \Lambda^{-1}(\alpha) = \{\mu\in\m: \Lambda(\mu)=\alpha\} \,,$$
and for a continuous function $\phi:X\to\R$ we write 
$$
X_\alpha(\phi) = \left\{x\in X: \lim_{n\to\infty} \frac{1}{n}\sum_{i=0}^{n-1}\phi(T^ix) = \alpha \right\}\,,
$$
$$\m_\alpha(\phi) = \left\{\mu\in\m: \int \phi\, d\mu =\alpha\text{ and }\Lambda(\mu)<\infty\right\}\,,$$
so that $X_\alpha=X_\alpha(\log|T'|)$ and $\m_\alpha=\m_\alpha(\log|T'|)$.
\end{notation}

We shall be interested in the Hausdorff dimension
(denoted $\dim_H$) of the level sets $X_\alpha$, for $\alpha\in A$.
Recall (see e.g.~\cite{falconer, pesinbook}) that
 $\dim_H(X_\alpha):= \inf\{\delta \hbox{ : } H^\delta(X_\alpha) = 0\}$,
where
$$
H^\delta(X_\alpha) := 
\lim_{\epsilon \to 0} \inf
\left\{
\sum_i
\text{diam}(U_i)^\delta :  \{U_i\} \text{ is an open cover of }X_\alpha \text{ with each }
\text{diam}(U_i) \leq \epsilon
\right\}
\,.
$$

\begin{defn}
The \emph{Lyapunov spectrum}\footnote{Also sometimes referred to as the \emph{multifractal spectrum of the Lyapunov exponent}.} is the function $L:A\to\R$ defined by
$$
L(\alpha) = \dim_H(X_\alpha)\,.
$$
\end{defn}


\begin{remark}\label{samelengthremark}
In the special case that all the intervals $X_i$ have equal length, the modulus of the derivative
of the piecewise linear map $T$ is constant, so the domain $A$ is a singleton,
and the Lyapunov spectrum $L:A\to\R$ is consequently a constant.
\end{remark}

\begin{prop} \cite{weiss}
The Lyapunov spectrum is real analytic on $\mathring{A} = (\alpha_{\min},\alpha_{\max})$. 
\end{prop}

\begin{defn}
A point $\alpha\in \mathring{A}=(\alpha_{\min},\alpha_{\max})$ 
which is a point of inflection of $L$ (i.e.~such that $L''(\alpha)=0$)
will be called a \emph{Lyapunov inflection}.
\end{defn}

\medskip

\subsection{Characterisations of the Lyapunov spectrum}\label{characterisationssubsection} \hfill\\

\begin{notation}
For a measure $\mu\in\m$, let $h(\mu)$ denote its \emph{entropy}.
We refer to $h:\m\to\R$ as the \emph{entropy map}. 
\end{notation}

The Lyapunov spectrum $L$ admits the following characterisation in terms of entropy:

\begin{prop}\label{Lcharacterisation1prop}
For a piecewise linear map,
if $\alpha\in A$
then
\begin{equation}\label{Lcharacterisation1}
L(\alpha) = \frac{1}{\alpha} \max_{\mu\in \m_\alpha} h(\mu)\,.
\end{equation}
\end{prop}
\begin{proof}
The identity
\begin{equation}\label{generaldimensionformula}
\text{dim}_H X_\alpha(\phi) = \max_{\mu\in\m_\alpha(\phi)} \frac{h(\mu)}{\Lambda(\mu)}
\end{equation}
was established in \cite[Thm.~1.1]{flw}.

In the special case $\phi=\log |T'|$, with $\int \phi\, d\mu = \Lambda(\mu)$, then $\m_\alpha(\phi)=\m_\alpha$
and $\Lambda(\mu)=\alpha$ for all $\mu\in \m_\alpha(\phi)=\m_\alpha$, and $X_\alpha(\phi) = X_\alpha$, therefore
(\ref{generaldimensionformula}) yields the required identity
$
L(\alpha) = \text{dim}_H X_\alpha =
 \frac{1}{\alpha} \max_{\mu\in \m_\alpha} h(\mu)\,.
 $
\end{proof}

\pagebreak

\begin{remark}
\item[\, (a)]
The characterisation (\ref{Lcharacterisation1}) was implicit in the work of Weiss \cite{weiss}, and appeared explicitly in Kesseb\"ohmer \& Stratmann \cite{ks} in the setting of the continued fraction map (cf.~the discussion in
\cite[p.~539]{iommikiwi}). It was generalised
in  \cite[Thm.~1.3]{iommijordan} to  a wider class of countable branch expanding maps
(see also \S \ref{infinitebranchsection}).
\item[\, (b)]
The characterisation (\ref{Lcharacterisation1}) shows in particular that 
\begin{equation}\label{concaveoveridentity}
L(\alpha)=\frac{C(\alpha)}{\alpha}
\end{equation}
 for a concave map $C$
(since the entropy map $h$ is affine \cite[Thm.~8.1]{walters},
 $C(\alpha)= \max_{\mu\in \m_\alpha} h(\mu)$ is concave).
\end{remark}

On the interior $\mathring{A} = (\alpha_{\min},\alpha_{\max})$ of the domain,
the Lyapunov spectrum $L$ can be characterised in terms of the equilibrium measures associated to potential
functions of the form $t\log|T'|$, $t\in\R$.
To make this precise, we define the following two maps $m$ and $\tau$:

\begin{notation}
Define $m:\R\to\m$ by letting $m(t)$ be the equilibrium measure
 for $t\log|T'|$,
i.e.~the unique measure which maximizes the quantity $h(\mu)+\int \log|T'|\, d\mu$ over all $\mu\in\m$
(see e.g.~\cite{parrypollicott, ruelle, walters}).

Define $\tau:\mathring{A}\to \R$ by
$$\tau=(\Lambda\circ m)^{-1}\,,$$
i.e.~$\tau(\alpha)$ is the unique real number such that $m(\tau(\alpha))$ has Lyapunov exponent equal to $\alpha$.
\end{notation}

Then:

\begin{prop} \label{Lcharacterisation2prop}
For a piecewise linear map,
if $\alpha\in \mathring{A} = (\alpha_{\min},\alpha_{\max})$ then
$$
L(\alpha) = \frac{1}{\alpha}  h(m(\tau(\alpha)))\,.
$$
That is, on the  interior $\mathring{A} = (\alpha_{\min},\alpha_{\max})$ of the domain,
\begin{equation}\label{Lcharacterisation2}
L = \frac{h\circ m\circ \tau}{id}\,.
\end{equation}
\end{prop}
\begin{proof}
This follows from Proposition \ref{Lcharacterisation1prop} (i.e.~the characterisation (\ref{Lcharacterisation1})), together with the simple fact that
$\max_{\mu\in \m_\alpha} h(\mu) = h(m(\tau(\alpha)))$ 
(which holds because $h(m(\tau(\alpha)))+\tau(\alpha)\alpha \ge h(\mu)+\tau(\alpha)\Lambda(\mu)$
for all $\mu\in\m$, thus
$h(m(\tau(\alpha)))+\tau(\alpha)\alpha \ge h(\mu)+\tau(\alpha)\alpha$
for all $\mu\in\m_\alpha$).
\end{proof}

\begin{notation}
Recall that for a general continuous function $\phi:X\to\R$, the \emph{pressure} $P(\phi)$ is defined
(see e.g.~\cite{parrypollicott, ruelle, walters}) by 
\begin{equation}\label{pressurelimitdefn}
P(\phi) = \lim_{n\to\infty}
\frac{1}{n}\log \sum_{T^n(x)=x} e^{\phi(x)+\phi(Tx)+\ldots +\phi(T^{n-1}x)}\,,
\end{equation}
and admits the well known alternative characterisation
$$P(\phi)=\max_{\mu\in \m} \left( h(\mu) +\int \phi\, d\mu \right) \,.$$
Now define
 $p:\R\to\R$ by 
$$p(t)=P(t \log|T'|)\,.$$ 
\end{notation}

The function $p$ is $C^\omega$ and strictly convex, so its derivative
$p'$ is an orientation-preserving $C^\omega$ diffeomorphism onto its image.
Clearly $p(t) = P(t \log|T'|) = h(m(t))+\int t \log|T'|\, dm(t)$, so in particular
$p(\tau(\alpha)) = h(m(\tau(\alpha))) + \alpha \tau(\alpha)$ for all $\alpha\in\mathring{A}$, in other words
\begin{equation}\label{pequation}
p\circ\tau = h\circ m\circ \tau + id.\tau\quad\text{on }\mathring{A}\,.
\end{equation}
We deduce the following result (which is well known, see e.g.~\cite[Eq.~(3), p.~539]{iommikiwi} which differs superficially due to usage of $P(-t\log|T'|)$ rather than the $P(t\log|T'|)$ considered here):

\begin{prop}
For a piecewise linear map,
on 
the  interior 
$\mathring{A} = (\alpha_{\min}, \alpha_{\max})$, 
the Lyapunov spectrum $L$ can be written as
\begin{equation}\label{Lcharacterisation3}
L = \frac{p\circ \tau}{id} - \tau\,.
\end{equation}
\end{prop}
\begin{proof}
This follows from Proposition \ref{Lcharacterisation2prop} 
together with the identity (\ref{pequation}).
\end{proof}

The $C^\omega$ diffeomorphism
 $p':\R\to(\alpha_{\min},\alpha_{\max})$ 
satisfies (see e.g.~\cite[p.~60]{parrypollicott}, \cite[p.~133]{ruelle})
\begin{equation}\label{pderivformula}
p'(t) = \int \log|T'|\, dm(t) = \Lambda(m(t))\,,
\end{equation}
so in particular
$p'(\tau(\alpha)) = \Lambda(m(\tau(\alpha))) = \alpha$, in other words
\begin{equation}\label{inverses}
p' \circ \tau = id\,,
\end{equation}
so $p':\R\to \mathring{A}$ and $\tau:\mathring{A}\to\R$ are inverses of each other.
It follows that the Lyapunov spectrum $L$ can be expressed purely in terms of the function $p$ as follows:

\begin{prop} 
For a piecewise linear map,
on the interior $\mathring{A} = (\alpha_{\min}, \alpha_{\max})$, the Lyapunov spectrum $L$ can be expressed as
\begin{equation}\label{Lcharacterisation4}
L = \left( \frac{p}{p'} - id\right)\circ (p')^{-1} \,.
\end{equation}
\end{prop}
\begin{proof}
This follows directly from (\ref{Lcharacterisation3}) and (\ref{inverses}).
\end{proof}

\subsection{Formulae for derivatives of the Lyapunov spectrum}\label{derivativessubsection} \hfill\\

The identity (\ref{Lcharacterisation4}) yields the following formula for the derivative of the Lyapunov spectrum:

\begin{prop}\label{firstderivofLprop}
For a piecewise linear map,
on the interior $\mathring{A} = (\alpha_{\min}, \alpha_{\max})$, the first derivative $L'$ of the Lyapunov spectrum can be expressed as
\begin{equation}\label{Lfirstderiv}
L' =  \frac{- \, p\circ (p')^{-1}}{id^2}\,.
\end{equation}
In other words, $L'(\alpha) = -p( (p')^{-1}(\alpha))/\alpha^2$ for all $\alpha\in \mathring{A} = (\alpha_{\min}, \alpha_{\max})$.
\end{prop}
\begin{proof}
Now $L = \left( \frac{p}{p'} - id\right)\circ (p')^{-1}$ from (\ref{Lcharacterisation4}), and differentiating this identity yields
$$
L'=\left(  \frac{p}{p'} - id \right)' \circ (p')^{-1} . ( (p')^{-1})'
=
\left( \frac{-p\, p''}{(p')^2}\right)\circ (p')^{-1} . \frac{1}{p''\circ (p')^{-1}}
=
 \frac{- \, p\circ (p')^{-1}}{id^2} \,,
$$
as required.
\end{proof}

\begin{cor}
For a piecewise linear map,
the Lyapunov spectrum $L:A\to\R$ has precisely one critical point, namely
at 
$\alpha = p'(-\dim_H(X)) = \Lambda( m(-\dim_H(X)))$.
\end{cor}
\begin{proof}
From (\ref{Lfirstderiv}), $L'(\alpha)=0$ if and only if $p(\tau(\alpha))=0$,
since $\tau=(p')^{-1}$.
That is, $L'(\alpha)=0$ if and only if $P(\tau(\alpha)\log|T'|)=0$, and this occurs
if and only if $\tau(\alpha)=-\dim_H(X)$, by Bowen's pressure formula for the dimension of $X$
(see e.g.~\cite{bowen}, \cite{falconer}, \cite[Ch.~7]{pesinbook}).
Thus $L'(\alpha)=0$ if and only if $\alpha=\tau^{-1}(-\dim_H(X)) = p'(-\dim_H(X))=\int \log|T'|\, dm(-\dim_H(X))
=\Lambda(m(-\dim_H(X)))$.
\end{proof}

 We are now able to derive a formula for the second derivative of the Lyapunov spectrum:

\begin{prop}\label{2ndderivLprop}
For a piecewise linear map,
and for $\alpha\in \mathring{A}= (\alpha_{\min}, \alpha_{\max})$, the second derivative $L''$ of the Lyapunov spectrum
can be expressed as
\begin{equation}\label{Lsecondderiv}
L''(\alpha) = \frac{2p(t)p''(t) - p'(t)^2}{p''(t)p'(t)^3}
\end{equation}
where $t=(p')^{-1}(\alpha) = \tau(\alpha)$.
\end{prop}
\begin{proof}
From (\ref{Lfirstderiv}), $L'(\alpha) = -p( (p')^{-1}(\alpha))/\alpha^2$, and differentiation yields
$$
L''(\alpha)
=
-\left( \frac{ \alpha^2 \frac{p'( t)}{p''(t)} - 2\alpha p(t)}{\alpha^4}\right)\,.
$$
Now $p'(t)=\alpha$ (since $t=(p')^{-1}(\alpha)$), so
$$
L''(\alpha)
=
-\left( \frac{ p'(t)^2 \frac{p'( t)}{p''(t)} - 2p'(t) p(t)}{p'(t)^4}\right)
= \frac{2p(t)p''(t) - p'(t)^2}{p''(t)p'(t)^3}\,,
$$
as required.
\end{proof}

\begin{cor}
For a piecewise linear map,
the Lyapunov spectrum $L$ has a point of inflection
at $\alpha\in  \mathring{A}= (\alpha_{\min},\alpha_{\max})$ if and only if
\begin{equation}\label{inflectionequationnumerator}
2p(t)p''(t) = p'(t)^2\,,
\end{equation}
where $t=(p')^{-1}(\alpha) = \tau(\alpha)$.
\end{cor}

\begin{remark}
Formulae for the first and second derivatives of the Lyapunov spectrum in terms of entropy appear in
 \cite[\S 8]{iommi}. 
 Proposition \ref{2ndderivLprop} should be compared to the derivation on
 \cite[p.~544]{iommikiwi}\footnote{Note, however, that the exposition on \cite[p.~544]{iommikiwi} is not completely correct: line 13 should read as $\frac{1}{\alpha'(t)} \frac{2\alpha'(t)P(-t\log|T'|)+\alpha(t)^2}{\alpha(t)^3}$, and line 15 should read as
 $\alpha'(t) = -\sigma^2(t) = -P(-t\log|T'|)'' <0$.}.
\end{remark}

\subsection{Lyapunov inflections for piecewise linear maps}\label{pwlinearsubsection} \hfill\\

While the preceding analysis could have been stated in the more general setting of (non-linear) expanding maps, 
henceforth we use the fact that a piecewise linear map $T$ is such that
on each interval $X_i$ the absolute value of its derivative
 $|T'|$ is equal to $|X_i|^{-1}$ (i.e.~the reciprocal of the length of $X_i$). 
 In this case from (\ref{pressurelimitdefn}) we see that the pressure
   is given by
\begin{equation}\label{pressurepiecewiselinearcase}
p(t) = P(t\log|T'|) = \log \left( \sum_{i=1}^q |X_i|^{-t}\right) \,.
\end{equation}

For ease of notation in what follows, we introduce the following function $F$:

\begin{notation}
Define $F=F_T:\R\to\R$ by
\begin{equation}\label{generalFdefn}
F(t) = F_T(t) = \sum_{i=1}^q |X_i|^{t} \,.
\end{equation}
\end{notation}

Clearly
\begin{equation}\label{pF}
p(t) = \log F(-t)\quad\text{for all }t\in\R\,.
\end{equation}

\begin{prop}\label{keyequationprop}
For a piecewise linear map,
the corresponding Lyapunov spectrum $L$ has a point of inflection at $\alpha\in(\alpha_{\min},\alpha_{\max})$ if and only if
\begin{equation}\label{keyequation_initial}
\frac{1}{2\log F(s)} =  \frac{F''(s) F(s)}{F'(s)^2} -1
\end{equation}
where
$s= 
-(p')^{-1}(\alpha) = -\tau(\alpha)$,
and $F$ is given by (\ref{generalFdefn}).
\end{prop}
\begin{proof}
Rearranging (\ref{inflectionequationnumerator}) we see that $\alpha$ is an inflection point for $L$ if and only if
\begin{equation}\label{rearrange}
\frac{1}{2p(t)} = \frac{p''(t)}{p'(t)^2}
\end{equation}
where $t=(p')^{-1}(\alpha) = \tau(\alpha)$.

From (\ref{pF}), setting $s=-t = -(p')^{-1}(\alpha) = -\tau(\alpha)$ then $p(t) = \log F(-t)=\log F(s)$, so the lefthand side of (\ref{rearrange}) is equal to the lefthand side of 
(\ref{keyequation_initial}).
Differentiating (\ref{pF}), 
$$
p'(t)=\frac{-F'(-t)}{F(-t)}
$$
and
$$
p''(t) = \frac{F(-t)F''(-t)-F'(-t)^2}{F(-t)^2}\,,
$$
so
$$
\frac{p''(t)}{p'(t)^2} = \frac{F(-t)F''(-t)-F'(-t)^2}{F'(-t)^2} = \frac{F''(s) F(s)}{F'(s)^2} -1\,,
$$
in other words the righthand side of (\ref{rearrange}) is equal to the righthand side of 
(\ref{keyequation_initial}), and the proof is complete.
\end{proof}

\section{Two branch maps have at most two Lyapunov inflections}\label{twobranchsection}

Consider a piecewise linear map with two branches, i.e.~where $q= 2$.
If $|X_1|=|X_2|$ then the Lyapunov spectrum is constant (cf.~Remark \ref{samelengthremark}),
and in particular has no points of inflection.
If $|X_1|\neq |X_2|$ then without loss of generality $|X_1|>|X_2|$, and defining $a=|X_1|^{-1}$ and $b=|X_2|^{-1}$
(so that $b>a>1$), the domain is $A=[\log a, \log b]$.
As noted by
Iommi \& Kiwi \cite[p.~539]{iommikiwi}, 
the Lyapunov spectrum $L$
has the closed form expression
\begin{equation}\label{closedform2branch}
L(\alpha) = \frac{1}{\alpha}\left( 
- \left( \frac{\log b - \alpha}{\log(b/a)}\right)\log\left( \frac{\log b - \alpha}{\log(b/a)}\right)
\,
- 
\,
\left( \frac{\alpha - \log a}{\log(b/a)}\right)\log\left( \frac{\alpha - \log a}{\log(b/a)}\right) \right) \,,
\end{equation}
since 
$m(\tau(\alpha))$ is the Bernoulli measure giving mass
$ \frac{\log b - \alpha}{\log(b/a)}$ to $X_1$ and mass
$ \frac{\alpha - \log a}{\log(b/a)}$ to $X_2$,
with entropy
(see e.g.~\cite[Thm.~4.26]{walters}) equal to
$$- \left( \frac{\log b - \alpha}{\log(b/a)}\right)\log\left( \frac{\log b - \alpha}{\log(b/a)}\right)
- \left( \frac{\alpha - \log a}{\log(b/a)}\right)\log\left( \frac{\alpha - \log a}{\log(b/a)}\right)\,,
$$
so that (\ref{closedform2branch}) is a consequence of Proposition \ref{Lcharacterisation2prop}.

Iommi \& Kiwi \cite[p.~539]{iommikiwi} 
conjectured that, in the setting of a two branch expanding map,
the Lyapunov spectrum has at most two points of inflection.
In fact the number of such inflections is necessarily even, since $L$ is concave
on some neighbourhood $[\log a, \gamma]$ of the left endpoint of $A$, and some neighbourhood $[\delta,\log b]]$ of the right endpoint of $A$ (see \cite[p.~539]{iommikiwi}).
In the case that the map is piecewise linear, we are able to answer Iommi \& Kiwi's conjecture in the affirmative:

\begin{theorem}\label{lyapunov2inflections}
For a 2-branch piecewise linear map, 
the Lyapunov spectrum $L$ is either concave 
or has precisely two inflection points. 
\end{theorem}

The following more precise characterisation follows from Theorem \ref{lyapunov2inflections} and \cite[Thm.~A]{iommikiwi}:

\begin{theorem}\label{2branchquantitativetheorem}
For a 2-branch piecewise linear map,
where without loss of generality $|X_1|\ge|X_2|$,
the Lyapunov spectrum $L$ is concave 
if 
$$
\frac{\log |X_1|}{\log |X_2|} \le \frac{\sqrt{2\log 2} +1}{\sqrt{2\log 2} -1} \approx 12.2733202 \,,
$$
and otherwise has precisely two inflection points. 
\end{theorem}

Theorem \ref{lyapunov2inflections} is in fact a consequence of the following more general result:

\begin{theorem}\label{generaltheorem}
Suppose the $C^3$ function $\phi:[0,1]\to\R$ is negative, convex, symmetric about the point $1/2$,
and such that the third derivative $\phi'''<0$ on $(0,1/2)$ (hence $\phi'''>0$ on $(1/2,1)$).
Then for $c>1$, the function $M$ defined by 
$$M(x)=\frac{\phi(x)}{x-c}$$
has at most two points of inflection (i.e.~at most two zeros of its second derivative $M''$).
\end{theorem}


To see that Theorem \ref{lyapunov2inflections} follows from Theorem \ref{generaltheorem}, note that 
if we introduce $$x= \frac{\log b - \alpha}{\log(b/a)}$$
then the
Lyapunov spectrum $L$ in (\ref{closedform2branch}) can be written as
$$
L(\alpha) = 
-\frac{\phi(x)}{\alpha}
$$
where
\begin{equation}\label{h}
\phi(x)=x\log x+ (1-x)\log(1-x)\,.
\end{equation}
Since $-\alpha= x\log(b/a) - \log b$ then
$$
L(\alpha) = 
\frac{1}{\log(b/a)} \frac{\phi(x)}{x-c}\,,
$$
where 
$$c=\frac{\log b}{\log(b/a)} > 1\,,$$
and if we define 
$$M(x)=\frac{\phi(x)}{x-c}$$
then
$$
L(\alpha) = 
\frac{M(x)}{\log(b/a)}\,,$$
and it is worth recording the following easy lemma:

\begin{lemma}
The functions $L$ and $M$ have the same number of points of inflection.
\end{lemma}
\begin{proof}
Now $M(x)=kL(\alpha(x))$, where $k=\log(b/a)$ and $\alpha(x)=-kx + \log b$,
and
$$M''(x)=k^3L''(\alpha(x))\,,$$ 
so $x_0$ satisfies $M''(x_0)=0$ if and only if
$L''(\alpha(x_0))=0$. 
The affine function $\alpha(x)=-x\log(b/a) + \log b$ is in particular a bijection (from $[0,1]$ to $[\log a,\log b]$), so there
is a one-to-one correspondence between 
zeros of $M''$ and zeros of $L''$, as required.
 \end{proof}

The fact that
Theorem \ref{lyapunov2inflections} follows from Theorem \ref{generaltheorem} is then a consequence of the following 
properties of the function $\phi$ defined in (\ref{h}).

\begin{lemma}
The function $\phi:[0,1]\to\R$ defined by
$$\phi(x)=x\log x+ (1-x)\log(1-x)$$
has the following properties:
\item[\, (i)] $\phi(x)=\phi(1-x)$ for all $x\in[0,1]$,
\item[\, (ii)] $\phi\le 0$ on $[0,1]$, and $\phi<0$ on $(0,1)$,
\item[\, (iii)] $\phi'<0$ on $(0,1/2)$ and $\phi'>0$ on $(1/2,1)$,
\item[\, (iv)] $\phi''\ge 4 >0$ on $[0,1]$,
\item[\, (v)]  $\phi'''<0$ on $(0,1/2)$ and $\phi'''>0$ on $(1/2,1)$,
\end{lemma}
\begin{proof}
Properties (i) and (ii) are easily seen to hold.
The derivative formula
$$
\phi'(x)=\log\left(\frac{x}{1-x}\right)
$$
yields property (iii),
the second derivative formula
$$
\phi''(x)=\frac{1}{x(1-x)}
$$
yields property (iv),
and the third derivative formula
$$
\phi'''(x)=\frac{2x-1}{x^2(1-x)^2}
$$
yields property (v).
\end{proof}


Now we prove Theorem \ref{generaltheorem}, concerning the inflection points of the function
$$M(x)=\frac{\phi(x)}{x-c}\,.$$

Note that
$$
M'(x) = \frac{\phi'(x)}{x-c} - \frac{\phi(x)}{(x-c)^2}\,,
$$
and so
$$
M''(x) = \frac{\phi''(x)}{x-c} - \frac{2\phi'(x)}{(x-c)^2} + \frac{2\phi(x)}{(x-c)^3}\,.
$$

So if $x_0$ is an inflection point of $M$ then $M''(x_0)=0$, and therefore $(x_0-c)^3M(x_0)=0$,
in other words
$$
\phi''(x_0)(x_0-c)^2-2\phi'(x_0)(x_0-c)+2\phi(x_0)=0\,,
$$
hence necessarily either
\begin{equation}\label{positive}
x_0 - c = \frac{\phi'(x_0) + \sqrt{\phi'(x_0)^2 - 2\phi(x_0)\phi''(x_0)}}{\phi''(x_0)}
\end{equation}
or
\begin{equation}\label{negative}
x_0 - c = \frac{\phi'(x_0) - \sqrt{\phi'(x_0)^2 - 2\phi(x_0)\phi''(x_0)}}{\phi''(x_0)} \,.
\end{equation}

If we now define functions $\Phi^+$ and $\Phi^-$ by
\begin{equation}\label{H+defn}
\Phi^+(x) = -x + \frac{\phi'(x) + \sqrt{\phi'(x)^2 - 2\phi(x)\phi''(x)}}{\phi''(x)}
\end{equation}
and
\begin{equation}\label{H-defn}
\Phi^-(x) = -x + \frac{\phi'(x) - \sqrt{\phi'(x)^2 - 2\phi(x)\phi''(x)}}{\phi''(x)}\,,
\end{equation}
then we see that (\ref{positive}) and (\ref{negative}) are respectively equivalent to
\begin{equation}\label{positiveH}
\Phi^+(x_0) = -c
\end{equation}
and
\begin{equation}\label{negativeH}
\Phi^-(x_0) = -c \,.
\end{equation}

That is, if $x_0$ is an inflection point of $M$ then necessarily either
(\ref{positiveH}) or (\ref{negativeH}) holds.

We now show that in fact (\ref{positiveH}) can never hold:

\begin{lemma}\label{negativeonly}
If $x_0$ is an inflection point of $M$ then necessarily (\ref{negativeH}) holds.
\end{lemma}
\begin{proof}
In view of the above discussion it suffices to show that (\ref{positiveH}) can never hold.
For this, note that an 
assumption of Theorem \ref{generaltheorem} is that $c>1$, so that $-c<-1$. We claim that the image of $\Phi^+$ is disjoint from $(-\infty, -1)$, so that no $x_0$ can satisfy $\Phi^+(x_0)=-c$.

Now $\phi\le0$ on $[0,1]$, and $\phi''>0$ on $[0,1]$, so $\phi(x)\phi''(x)\le0$ for all $x\in[0,1]$.
Therefore $\phi'(x)^2-2\phi(x)\phi''(x) \ge \phi'(x)^2$, and consequently
$$
\sqrt{\phi'(x)^2-2\phi(x)\phi''(x) } \ge -\phi'(x)\ \text{for all $x\in[0,1]$,}
$$
in other words
$$
\phi'(x)+\sqrt{\phi'(x)^2-2\phi(x)\phi''(x) } \ge 0 \ \text{for all $x\in[0,1]$,}
$$
and hence
$$
\frac{\phi'(x) + \sqrt{\phi'(x)^2 - 2\phi(x)\phi''(x)}}{\phi''(x)} \ge 0 \ \text{for all $x\in[0,1]$,}
$$
since $\phi''\ge0$ on $[0,1]$.
It follows from the formula (\ref{H+defn}) that $\Phi^+(x) \ge -x$, and hence that $\Phi^+\ge -1$ on $[0,1]$, as required.
\end{proof}

Having eliminated the need to consider the function $\Phi^+$, we now simplify our notation by defining $\Phi:=\Phi^-$,
in other words we set
\begin{equation}\label{Hdefn}
\Phi(x) = \Phi^-(x) = -x + \frac{\phi'(x) - \sqrt{\phi'(x)^2 - 2\phi(x)\phi''(x)}}{\phi''(x)}\,,
\end{equation}
and from Lemma \ref{negativeonly} we know that if
$x_0$ is an inflection point of $M$ then necessarily
\begin{equation}\label{H}
\Phi(x_0) = -c \,.
\end{equation}

To conclude the proof of Theorem \ref{generaltheorem}, it now suffices to show
(in Lemma \ref{increasingdecreasinglemma} below) that $\Phi$ is strictly
decreasing on $(0,1/2)$, and strictly increasing on $(1/2,1)$, since
from the above discussion it follows that 
any point $-c\in (- \infty, -1)$ has at most two $\Phi$-preimages, and hence that $M$ has at most two points of inflection.


\begin{lemma}\label{increasingdecreasinglemma}
Under the hypotheses on $\phi$ of Theorem \ref{generaltheorem},
the function $\Phi$ defined by (\ref{Hdefn}) is strictly decreasing on $(0,1/2)$, and strictly increasing on $(1/2,1)$.
\end{lemma}
\begin{proof}
It will be shown that
$\Phi'$ is strictly negative on $(0,1/2)$, and strictly positive on $(1/2,1)$.

The formula
$$
\Phi' = -1 +\frac{1}{\phi''}\left( \phi''+ \frac{\phi \phi'''}{\sqrt{(\phi')^2-2\phi \phi''}}\right) - \frac{\phi'''}{(\phi'')^2}\left(\phi' - \sqrt{(\phi')^2-2\phi\phi''}\right) 
$$
simplifies to
$$
\Phi' = \frac{\phi'''}{(\phi'')^2}\left( \frac{\phi\phi''}{\sqrt{(\phi')^2-2\phi\phi''}} - \phi' + \sqrt{(\phi')^2-2\phi\phi''}\right)
$$
and since by assumption $\phi'''$ is negative on $(0,1/2)$ and positive on $(1/2,1)$, it will suffice to prove that
\begin{equation}\label{first}
 \frac{\phi\phi''}{\sqrt{(\phi')^2-2\phi\phi''}} - \phi' + \sqrt{(\phi')^2-2\phi\phi''} > 0\,.
\end{equation}
Rewriting (\ref{first}) as
\begin{equation*}\label{second}
\frac{\phi\phi''}{\sqrt{(\phi')^2-2\phi\phi''}} > \phi' - \sqrt{(\phi')^2-2\phi\phi''} \,,
\end{equation*}
and then as
\begin{equation*}\label{third}
\phi\phi'' > \phi'\sqrt{(\phi')^2-2\phi\phi''} - \left((\phi')^2-2\phi\phi''\right) \,,
\end{equation*}
it can then be rearranged as
\begin{equation}\label{fourth}
(\phi')^2 - \phi\phi'' > \phi'\sqrt{(\phi')^2-2\phi\phi''}  \,.
\end{equation}
The lefthand side of (\ref{fourth}) is positive on $[0,1]$ (since $\phi<0$ and $\phi''>0$), so (\ref{fourth}) is certainly true if its righthand side is negative (i.e.~on the sub-interval $[0,1/2]$).
If the righthand side of (\ref{fourth}) is positive then
squaring both sides gives
\begin{equation*}\label{fifth}
(\phi')^4 + (\phi\phi'')^2  - 2(\phi')^2\phi\phi''> (\phi')^2 \left((\phi')^2-2\phi\phi''\right)   \,,
\end{equation*}
or in other words
\begin{equation*}\label{sixth}
(\phi')^4 + (\phi\phi'')^2  - 2(\phi')^2\phi\phi''> (\phi')^4 -2(\phi')^2\phi\phi''   \,,
\end{equation*}
which simplifies to become
\begin{equation*}\label{seventh}
(\phi\phi'')^2 > 0   \,,
\end{equation*}
which is clearly true, so the result is proved.
\end{proof}




\section{The number of Lyapunov inflections is unbounded}\label{manyinflectionssection}

In this section we define a particular sequence of piecewise linear maps $T_N$, with the property that the number
of Lyapunov inflections of $T_N$ tends to infinity as $N\to\infty$.

\begin{defn}\label{TNdefn}
For integers $N\ge 6$, define
$$
q_N:= \sum_{j=6}^N 2^{j^2}\,
$$
Let $\X_N := \{X_i\}_{i=1}^{q_N}$ be a collection of pairwise disjoint closed sub-intervals of $[0,1]$, where
for each $6\le j\le N$, exactly
$2^{j^2}$ of the intervals have length equal to 
$2^{-2^j}$.
Let $T_N$ be a corresponding piecewise linear map.
\end{defn}

\begin{theorem}\label{finiteinflectionstheorem}
For $N\ge 27$, the
Lyapunov spectrum for the
 piecewise linear map $T_N$ has at least
$2(N-26)$ points of inflection.
\end{theorem}

\begin{remark}
\item[\, (a)] The Lebesgue measure of $\cup_{i=1}^{q_N} X_i$ is $\sum_{j=6}^N 2^{j^2-2^j}$, which is strictly smaller than 1 for all $N\ge 6$ (note that $\sum_{j=6}^N 2^{j^2-2^j} < 10^{-8}$), so it is certainly possible to choose the $X_i$ to be pairwise disjoint and contained in $[0,1]$.

\item[\, (b)] We prescribe the \emph{lengths} of the intervals in the collection $\X_N$, but need 
no further information about the intervals themselves (beyond the fact that they are pairwise disjoint, and contained in $[0,1]$), since translating various of the $X_i$ does not change the Lyapunov spectrum.
 Clearly it could be arranged that 
$\X_N\subset \X_{N+1}$ for all $N\ge 6$, which would lend the interpretation of $T_{N+1}$ \emph{evolving from} the 
preceding map $T_N$ (as described in \S \ref{introsection}) by adjoining $2^{(N+1)^2}$ new branches.

\item[\, (c)] The number of branches $q_N$ of $T_N$ is large. For example $T_6$ has 
$2^{36}=68,719,476,736$ branches, with $|T_6'|=2^{64}=18,446,744,073,709,551,616$ on each branch.
In Theorem \ref{finiteinflectionstheorem}, the smallest value of $N$ yielding more than two Lyapunov inflections is $N=28$,
and the map $T_{28}$ has $q_{28}>10^{236}$ branches.

\end{remark}

\begin{notation}
Following\footnote{The fact that various of the intervals in $\X_N$ have identical lengths
allows the representation (\ref{Fdef}) as a sum over the range $6\le j\le N$,
rather than over $6\le j\le q_N$.}
 the notation of (\ref{generalFdefn}), define $F_N:\R\to\R$ by 
\begin{equation}\label{Fdef}
F_N(s)=\sum_{j=6}^N 2^{j^2}2^{- 2^js} =  \sum_{j=6}^N 2^{j^2- 2^js} \,.
\end{equation}
Define
$$
U_j(s):=j^2-2^j s\,,
$$
so that
\begin{equation}\label{Fsum}
F_N = \sum_{j=6}^N 2^{U_j}\,.
\end{equation}

For $j\ge 1$ define
$$
s_j:= \frac{2j+1}{2^j} \,,
$$
and for $j\ge 2$ define the midpoint
$$m_j := \frac{s_j + s_{j-1}}{2}
= \frac{6j-1}{2^{j+1}}
\,,$$
so in particular
$$
s_1>m_2>s_2>m_3>s_3>\ldots
$$
and 
$$\lim_{j\to\infty} s_j = \lim_{j\to\infty}m_j = 0\,.$$
\end{notation}

\begin{remark}
The $s_j$ are defined so that $U_j(s_j)$ and $U_{j+1}(s_j)$ are equal, more precisely
$$U_j(s_j) =U_{j+1}(s_j)= j^2-2j-1\,,$$
and each $m_j$ is the mid-point of $s_j$ and $s_{j-1}$, with
$$
U_j(m_j) = j^2-3j+1/2\,.
$$
\end{remark}

In light of Proposition
\ref{keyequationprop},
to prove Theorem \ref{finiteinflectionstheorem}
 it suffices to establish the following result:

\begin{prop}\label{finitesolutionsprop}
For $N\ge 27$, the equation
\begin{equation}\label{keyequationN}
\frac{1}{2\log F_N(s)} =  \frac{F_N''(s) F_N(s)}{F_N'(s)^2} -1
\end{equation}
has at least $2(N-26)$ distinct solutions.
\end{prop}
\begin{proof}
Introducing the auxiliary functions
$$
G_N(s):= \frac{1}{2\log F_N(s)}
\ ,\quad$$
$$
H_N(s):= \frac{F_N(s)F_N''(s)}{F_N'(s)^2} -1\,,
$$
we claim that
\begin{equation}\label{GdominatesH}
G_N(m_k) > H_N(m_k)\quad\text{for all } \, 26\le k\le N\,,
\end{equation}
and
\begin{equation}\label{HdominatesG}
G_N(s_k) < H_N(s_k) \quad\text{for all } \, 26\le k\le N-1 \,.
\end{equation}

Note that the proposition 
will follow from 
(\ref{GdominatesH}) and (\ref{HdominatesG}), since 
the intermediate value theorem then guarantees that 
for each $26\le k\le N-1$
there is a solution to
(\ref{keyequationN}) in both of the intervals $(m_{k+1},s_k)$ and $(s_k,m_k)$.

To prove (\ref{GdominatesH}) and (\ref{HdominatesG}), first note
that the derivative of $F_N$ can be written as
\begin{equation}\label{Fprimesum}
F_N' = - (\log 2)\sum_{j=6}^N 2^{j+U_j}\,,
\end{equation}
and its second derivative as 
\begin{equation}\label{Fdoubleprimesum}
F_N'' = (\log 2)^2 \sum_{j=6}^N 2^{2j+U_j}
= (\log 2)^2 \sum_{j=6}^N 2^{V_j}\,,
\end{equation}
where we set 
$$V_j(s):=2j + U_j(s) = 2j+j^2-2^j s\,.$$

First we prove (\ref{GdominatesH}).
The strategy for this will be to approximate each of $F_N(m_k)$, $F_N'(m_k)$, $F_N''(m_k)$ by the single $j=k$ term in the respective sums (\ref{Fsum}), (\ref{Fprimesum}), (\ref{Fdoubleprimesum}).
In particular, we claim that if $c=16/5$ then, for all $26\le k\le N$,
\begin{equation}\label{Finequality}
F_N(m_k) < 2^{U_k(m_k)} (1+ c 2^{-k/2})
\end{equation}
and
\begin{equation}\label{Fdoubleprimeinequality}
F_N''(m_k) < (\log 2)^2\, 2^{2k+U_k(m_k)} (1+ c 2^{-k/2})\,,
\end{equation}
and combine these with the obvious lower bound
\begin{equation}\label{Fprimeinequality}
F_N'(m_k)^2 >  (\log 2)^2\, 2^{2k+2U_k(m_k)}
\end{equation}
to give
\begin{equation}\label{Hestimate}
H_N(m_k)  = \frac{F_N(m_k)F_N''(m_k)}{F_N'(m_k)^2} - 1
<
(1+ c 2^{-k/2})^2 -1
=2c\, 2^{-k/2}+c^2 2^{-k}\,, 
\end{equation}
whereas (\ref{Finequality}) yields
$$
\log F_N(m_k) < (\log 2) U_k(m_k) + \log(1+c2^{-k/2})
< (\log 2)(k^2-3k+1/2)+c2^{-k/2}
$$ 
and therefore
\begin{equation}\label{Gestimate}
G_N(m_k)=\frac{1}{2 \log F_N(m_k)} > \frac{1}{ (2\log 2)(k^2-3k+1/2)+c2^{1-k/2}}\,.
\end{equation}
Since $c=16/5$,
from (\ref{Hestimate}) and (\ref{Gestimate})
it is readily verified that
$G_N(m_k) > H_N(m_k)$ for\footnote{The appearance of the number 26 in 
Theorem \ref{finiteinflectionstheorem} and Proposition \ref{finitesolutionsprop} is due to 
this verification step.}
 $26\le k\le N$, as required.

To establish (\ref{Finequality}) and (\ref{Fdoubleprimeinequality}) let us first write
\begin{equation*}
F_N(m_k) = 2^{U_k(m_k)} \left( \sum_{j=6}^{k-1} 2^{U_j(m_k)-U_k(m_k)} + 1 + \sum_{j=k+1}^N 2^{U_j(m_k)-U_k(m_k)} \right)\,,
\end{equation*}
$$
F_N''(m_k) = (\log 2)^2\, 2^{V_k(m_k)} \left( \sum_{j=6}^{k-1} 2^{V_j(m_k)-V_k(m_k)} + 1 + \sum_{j=k+1}^N 2^{V_j(m_k)-V_k(m_k)} \right)\,.
$$
Now $V_j(m_k)-V_k(m_k) = 2(j-k) + U_j(m_k)-U_k(m_k)$, so
$V_j(m_k)-V_k(m_k)$ is smaller (respectively larger) than $U_j(m_k)-U_k(m_k)$ when
$j$ is smaller (respectively larger) than $k$, so
\begin{equation}\label{Ffirst}
F_N(m_k) < 2^{U_k(m_k)} \left( \sum_{j=6}^{k-1} 2^{U_j(m_k)-U_k(m_k)} + 1 + \sum_{j=k+1}^N 2^{V_j(m_k)-V_k(m_k)} \right)\,,
\end{equation}
and
\begin{equation}\label{Fdoubleprimefirst}
F_N''(m_k) = (\log 2)^2\, 2^{V_k(m_k)} \left( \sum_{j=6}^{k-1} 2^{U_j(m_k)-U_k(m_k)} + 1 + \sum_{j=k+1}^N 2^{V_j(m_k)-V_k(m_k)} \right)\,.
\end{equation}
It therefore remains to derive upper bounds on
\begin{equation}\label{headsum}
 \sum_{j=6}^{k-1} 2^{U_j(m_k)-U_k(m_k)}
 \end{equation}
 and
\begin{equation}\label{tailsum}
\sum_{j=k+1}^\infty 2^{V_j(m_k)-V_k(m_k)}\,.
\end{equation}

To bound (\ref{headsum}), writing $j=k-i$ for $1\le i\le k-6$ we calculate
$$
U_{k-i}(m_k)-U_k(m_k)
= -k/2+ (-2i+7/2-3\cdot 2^{-i})k +i^2+2^{-1-i}-1/2 \,,
$$
and since $-2i+7/2-3\cdot 2^{-i}<0$ for $i\ge1$, and $k\ge i+6$, then
$$
U_{k-i}(m_k)-U_k(m_k)
<
-k/2+A(i)\,,
$$
where
$$
A(i):= (-2i+7/2-3\cdot 2^{-i})(i+6) +i^2+2^{-1-i}-1/2 \,,
$$
so
\begin{equation}\label{abound}
 \sum_{j=6}^{k-1} 2^{U_j(m_k)-U_k(m_k)}
 =
  \sum_{i=1}^{k-6} 2^{U_{k-i}(m_k)-U_k(m_k)}
  <
2^{-k/2}  \sum_{i=1}^{k-6} 2^{A(i)}
  <
2^{-k/2}  \sum_{i=1}^{\infty} 2^{A(i)} = a 2^{-k/2}
\end{equation}
where
$$
a = \sum_{i=1}^{\infty} 2^{A(i)} = 2^{3/4}+2^{-51/8}+2^{-277/16}+\ldots < 17/10\,.
$$

To bound (\ref{tailsum}), writing $j=k+i$ we have
$$
V_{k+i}(m_k)-V_k(m_k) 
=
-k/2 +(2i +7/2 -3\cdot 2^i)k + i^2+2i+2^{i-1}-1/2
$$
and since $k\ge 6$, and $2i +7/2 -3\cdot 2^i<0$ for $i\ge1$, then
$$
V_{k+i}(m_k)-V_k(m_k) 
< 
-k/2 + B(i) \,,
$$
where 
$$
B(i):= 6(2i +7/2 -3\cdot 2^i) + i^2+2i+2^{i-1}-1/2\,,
$$
so
\begin{equation}\label{bbound}
\sum_{j=k+1}^N 2^{V_j(m_k)-V_k(m_k)}
=
\sum_{i=1}^{N-k} 2^{V_{k+i}(m_k)-V_k(m_k) }
< 2^{-k/2} \sum_{i=1}^{N-k} 2^{B(i)} = b 2^{-k/2}
\end{equation}
where
$$
b=\sum_{i=1}^{N-k} 2^{ B(i)} \le 2^{1/2}+2^{-35/2}+2^{-137/2}+\ldots
< 3/2\,.
$$

Combining (\ref{abound}) and (\ref{bbound}) gives
\begin{equation}\label{cbound}
\sum_{j=6}^{k-1} 2^{U_j(m_k)-U_k(m_k)}
+
\sum_{j=k+1}^N 2^{V_j(m_k)-V_k(m_k)}
< c 2^{-k/2}
\end{equation}
where
\begin{equation*}
c= 
17/10 +3/2 = 16/5\,,
\end{equation*}
and substituting (\ref{cbound}) into
(\ref{Ffirst}) and (\ref{Fdoubleprimefirst}) yields the claimed
inequalities (\ref{Finequality}) and (\ref{Fdoubleprimeinequality}).
This completes the proof of (\ref{GdominatesH}).

To prove (\ref{HdominatesG}), assume that  $26\le k\le N-1$, and note that
\begin{equation}\label{lowerF}
F_N(s_k)>2^{U_k(s_k)} = 2^{k^2-2k-1}\,,
\end{equation}
 so
$$
G_N(s_k) = \frac{1}{2\log F_N(s_k)} < \frac{1}{(2\log 2)(k^2-2k-1)}\,,
$$
and therefore
\begin{equation}\label{Gtk}
G_N(s_k) \le G_N(s_{26}) < \frac{1}{1246\log 2} < \frac{1}{863} \,.
\end{equation}
We next claim that
\begin{equation}\label{lemma1eq}
H_N(s_k) = \frac{F_N''(s_k) F_N(s_k)}{F_N'(s_k)^2} - 1 > \frac{1}{10}\,,
\end{equation}
and note that combining (\ref{Gtk}) and (\ref{lemma1eq}) will establish (\ref{HdominatesG}).

To prove (\ref{lemma1eq}), note that
\begin{equation}\label{lowerF2}
\frac{F_N''(s_k)}{(\log 2)^2} >  \sum_{j=k}^{k+1} 4^j 2^{j^2- 2^js_k} =
 (4^k+4^{k+1}) 2^{k^2-2^k s_k} = 
 \frac{5}{2} 4^k 2^{k^2-2k}\,,
\end{equation}
so combining with (\ref{lowerF})
gives
\begin{equation}\label{numerator}
\frac{ F_N(s_k) F_N''(s_k)}{(\log 2)^2} >   \frac{5}{2} 4^k 2^{2k^2-4k} \,.
\end{equation}

We now establish an upper bound on
\begin{equation}\label{Fderivbound}
-\frac{F_N'(s_k)}{\log 2} = \sum_{j=1}^\infty 2^{j+j^2- 2^js_k}
\end{equation}
by combining the exact evaluation of the two-term sum
\begin{equation}\label{exact}
\sum_{j=k}^{k+1} 2^{j+j^2- 2^js_k} = 
(2^k+2^{k+1}) 2^{k^2-2^ks_k} = 3\cdot 2^k 2^{k^2-2k-1} =\frac{3}{2} 2^k 2^{k^2-2k}
\end{equation}
with upper bounds on the head
\begin{equation}\label{head}
\sum_{j=1}^{k-1} 2^{j+j^2- 2^js_k} 
\end{equation}
and tail
\begin{equation}\label{tail}
\sum_{j=k+2}^N 2^{j+j^2- 2^js_k}  \,.
\end{equation}

To estimate (\ref{head}), 
write $j=k-i$ to give
\begin{equation}\label{Dsum}
\sum_{j=1}^{k-1} 2^{j+j^2- 2^js_k}  = \sum_{i=1}^{k-1} 2^{k-i +(k-i)^2-2^{-i}(2k+1)}
=2^{k^2-2k-1/2} D(k)
\end{equation}
where
\begin{equation}
\label{D}
D(k) := \sum_{i=1}^{k-1} 2^{i^2-i-2^{-i}+1/2 -2k(i+2^{-i}-3/2)} 
= \sum_{i=1}^{k-1} 2^{n_k(i)} = 1+  \sum_{i=2}^{k-1} 2^{n_k(i)}
\end{equation}
and
$$
n_k(i):= i^2-i-2^{-i}+1/2 -2k(i+2^{-i}-3/2) \,,
$$
noting that $n_k(1)=0$.

Clearly $D(2)=1$ and $D(k)>1$.
We claim that
\begin{equation}\label{Dbound}
  D(k)  < \frac54\,.
\end{equation}
To establish (\ref{Dbound}) note that 
if $2\le i\le k-1$ then $n_k(i)\le n_k(2)$, therefore $2^{n_k(i)} \le 2^{n_k(2)}$, so
\begin{equation}\label{3rdsup}
\sum_{i=2}^{k-1} 2^{n_k(i)} \le (k-2)2^{n_k(2)} = (k-2)2^{9/4 -3k/2}\,.
\end{equation}
The righthand side of (\ref{3rdsup}) is decreasing in $k$, and we are assuming that $k\ge 26$, so
\begin{equation}\label{3rdsupsup}
 \sum_{i=2}^{k-1} 2^{n_k(i)} \le
 24 \times 2^{9/4 - 39} < 10^{-9} \,,
\end{equation}
therefore combining (\ref{3rdsupsup}) with (\ref{D}) easily gives the claimed bound (\ref{Dbound}).

 Combining (\ref{Dbound}) with (\ref{Dsum}) then gives
 \begin{equation}\label{headbound}
 \sum_{j=1}^{k-1} 2^{j+j^2- 2^js_k} < \frac54 2^{k^2-2k-1/2} = \frac{5}{4\sqrt{2}} 2^k 2^{k^2-2k}2^{-k}\,,
 \end{equation}
which is our desired bound on the head (\ref{head}).

We now wish to estimate the tail (\ref{tail}), and for this write $j=k+i$ so that
\begin{equation}\label{isum}
\sum_{j=k+2}^N 2^{j+j^2- 2^js_k} = \sum_{i=2}^{N-k} 2^{k+i+(k+i)^2-2^i(2k+1)}
=2^{k^2+k} \sum_{i=2}^{N-k} 2^{i+i^2-2^i+2k(i-2^i)}
\end{equation}
and then write
\begin{equation}\label{isum2}
 \sum_{i=2}^{N-k} 2^{i+i^2-2^i+2k(i-2^i)} 
 = 2^{2-4k}\left(1+\sum_{i=3}^{N-k} 2^{i+i^2-2^i-2 + 2k(2+i-2^i)}\right)\,.
 \end{equation}
 Now $k\ge1$, and $2+i-2^i<0$ for $i\ge3$, so $2k(2+i-2^i)\le 4+2i-2^{i+1}$, therefore
 $$
\sum_{i=3}^{N-k}  2^{i+i^2-2^i-2 + 2k(2+i-2^i)} < \sum_{i=3}^{N-k} 2^{i^2+3(i-2^i)+2}
\le 2^{-4}+2^{-18}+2^{-54}+\ldots < \frac{1}{8}\,,
$$
so using this estimate in (\ref{isum2}) 
gives
$$
 \sum_{i=2}^{N-k} 2^{i+i^2-2^i+2k(i-2^i)} < \frac98 2^{2-4k} \,,
 $$
and substituting into (\ref{isum}) yields
\begin{equation}\label{tailbound}
\sum_{j=k+2}^N 2^{j+j^2- 2^js_k}
<
\frac{9}{8}2^{k^2+k+2-4k} 
= \frac92 2^k 2^{k^2-2k} 2^{-2k} \,,
\end{equation}
which is our desired bound on the tail (\ref{tail}).
 
 Substituting
 into
(\ref{Fderivbound})
the identity (\ref{exact}), 
together with the bounds (\ref{headbound})
and  (\ref{tailbound}),
 gives
 \begin{equation*}\label{Fderivestimate}
-\frac{F_N'(s_k)}{\log 2} 
< 
2^k 2^{k^2-2k}\left(\frac32 + \frac{5}{4\sqrt{2}} 2^{-k}+ \frac92 2^{-2k} \right) \,,
 \end{equation*}
 and therefore
 \begin{equation}\label{denomFprime}
 \frac{F_N'(s_k)^2}{(\log 2)^2} < 4^k 2^{2k^2-4k} \left(\frac32 + \frac{5}{4\sqrt{2}} 2^{-k}+ \frac92 2^{-2k} \right)^2 \,.
 \end{equation}
 
 Combining (\ref{numerator}) and (\ref{denomFprime}) gives
 \begin{equation}\label{almostfinal}
 \frac{F_N(s_k) F_N''(s_k)}{F_N'(s_k)^2} > \frac{5}{2}\left(\frac32 + \frac{5}{4\sqrt{2}} 2^{-k}+ \frac92 2^{-2k} \right)^{-2}
 =\frac{10}{9}\left( 1+ \frac{5}{6\sqrt{2}} 2^{-k}+ 3\cdot 2^{-2k} \right)^{-2}\,.
 \end{equation}
 
Note that the righthand side of (\ref{almostfinal}) is increasing in $k$, and since $k\ge 26$ then
 \begin{equation*}
 \frac{F_N(s_k) F_N''(s_k)}{F_N'(s_k)^2} > \frac{10}{9}\left( 1+ \frac{5}{6\sqrt{2}} 2^{-26}+ 3\cdot 2^{-52} \right)^{-2}
 = 1.11111109159\ldots
 > \frac{11}{10}
 \,,
 \end{equation*}
and therefore
 (\ref{lemma1eq}) follows, as required.
\end{proof}

Although the map $T_N$ is defined so that its derivative is identical on various branches, it should be apparent that
this is a convenience rather than a strict condition.
The dependence of the Lyapunov spectrum (with respect to e.g.~some $C^k$ topology)
on the underlying map is continuous, so that 
for example a sufficiently small perturbation of $T_N$ to a piecewise linear map $\tilde T_N$ on a slightly different collection
of intervals
 $\tilde \X_N=\{\tilde X_i\}_{i=1}^{q_N}$,
 such that the lengths $|\tilde X_i|$ 
 (and hence the derivatives of the various branches)
 are all distinct,
will not change the number of Lyapunov inflections.

Similarly, $T_N$ could be perturbed in a nonlinear way, yielding a nearby nonlinear 
(full branch) expanding map with the same number of Lyapunov inflections.
More precisely, given $\X_N := \{X_i\}_{i=1}^{q_N}$
as in Definition \ref{TNdefn}, for any integer $k\ge 1$ let $C^k_N$ denote the set of maps 
$\cup_{i=1}^{q_N} X_i\to [0,1]$ such that the restriction to each $X_i$ is a $C^k$ diffeomorphism onto $[0,1]$,
equipped with the $C^k$ norm $\|T\|_k=\sup\{|T^{(i)}(x)|: x\in \cup_{i=1}^{q_N} X_i, 0\le i\le k\}$, where $T^{(i)}$ denotes the $i$-th order derivative of $T$.
As noted in Remark \ref{openingremark} (c), 
the definitions and basic properties of Lyapunov spectra detailed in subsections \ref{lyap2subsection}, \ref{characterisationssubsection} and \ref{derivativessubsection} are valid
for expanding members of $C^k_N$, in particular those that are sufficiently $C^k$-close to $T_N$, 
and we deduce:

\begin{theorem}
For all integers $k\ge 1$, $N\ge 27$, there is a non-empty open subset of $C^k_N$ whose elements all have at least $2(N-26)$ Lyapunov inflections.
\end{theorem}

\section{An infinite branch map with infinitely many Lyapunov inflections}\label{infinitebranchsection}

Using essentially the same strategy as in \S \ref{manyinflectionssection}, we can prove the following:

\begin{theorem}\label{infiniteinflectionsthm}
There is a countably infinite branch piecewise linear map whose Lyapunov spectrum
has a countable infinity of inflection points.
\end{theorem}

To prove Theorem \ref{infiniteinflectionsthm} we can define $\X=\{X_i\}_{i=1}^\infty$ to be a countable collection
of pairwise disjoint closed sub-intervals of $[0,1]$,
consisting of $2^{j^2}$ intervals of length $2^{-2^j}$ for all $j\ge 6$.
By analogy with Definition \ref{piecewiselinearmapdefn} we then define $T$ to be a full branch affine homeomorphism
onto $[0,1]$ on each interval in $\X$, with the analogous definition of its invariant set. 
Such a $T$ is in particular an \emph{expanding Markov-R\'enyi map} (see e.g.~\cite{iommijordan} for multifractal analysis in this general context),
and the collection of its inverse branches 
is a \emph{conformal iterated function system}
(see e.g.~\cite{hanusmauldinurbanski} for multifractal analysis in this general context).

The theory of \S \ref{preliminariessection} holds almost verbatim in this countable branch setting, the only significant difference being that
$$
F(t)  = \sum_{i=1}^\infty |X_i|^{t}
$$
is defined only for $t>0$,
and
$$
p(t) = \log F(-t)\quad\text{for all }t<0\,.
$$
The approach of Proposition \ref{finitesolutionsprop}
can then be used
to establish that
\begin{equation*}\label{keyequation}
\frac{1}{2\log F(s)} =  \frac{F''(s) F(s)}{F'(s)^2} -1
\end{equation*}
has infinitely many solutions, by again simply proving the analogues of (\ref{GdominatesH}) and
(\ref{HdominatesG}) for all $k\ge 26$, and this yields Theorem \ref{infiniteinflectionsthm}.

We also deduce the following:

\begin{cor}\label{infinitelymanyinflectionscorollary}
There is a countably infinite branch piecewise linear map such that
for all $k\ge 2$, its Lyapunov spectrum has
its $k^{th}$ order derivative equal to zero at infinitely many distinct points.
\end{cor}

\section{Examples with a specified number of Lyapunov inflections}\label{examplessection}

In this section we 
consider some specific piecewise linear maps with strictly more than two points of inflection in their
Lyapunov spectra.
By way of a complement to the explicit examples of Figures \ref{4figure} and \ref{6figure}, and the maps $T_N$ of \S
\ref{manyinflectionssection}, the cases considered here show that a given number of Lyapunov inflections
(namely 4, 6, and 8) can be attained for maps with relatively few branches
(namely 7, 62, and 821), an issue we return to in Question \ref{minbranches} below.
The maps are defined as follows.

\begin{example} (The map $S_4$)

\noindent
Let $\{X_i\}_{i=1}^7$ be a collection of 7 pairwise disjoint closed sub-intervals of $[0,1]$,
where $X_1$ has length $4/5$, the five intervals $X_2,\ldots, X_6$ all have length $1/65 = (4/5)(1/52)$,
and $X_7$ has length $1/3380 = (4/5)(1/52)^2$, and
let $S_4$ be any piecewise linear map associated to $\{X_i\}_{i=1}^7$.
The invariant set for $S_4$ has Hausdorff dimension equal to $0.8167\ldots$,
and its Lyapunov spectrum has domain $[\log(5/4),\log(3380)] = [ 0.2231\ldots, 8.1256\ldots]$.
\end{example}

\begin{example} (The map $S_6$)

\noindent
Let $\{X_i\}_{i=1}^{62}$ be a collection of 62 pairwise disjoint closed sub-intervals of $[0,1]$,
where one interval has length $4/5$, thirty intervals have length $1/400=(4/5)(1/320)$,
another thirty intervals have length $1/128000=(4/5)(1/320)^2$,
and the remaining interval has length $1/40960000 =(4/5)(1/320)^3$, and
let $S_6$ be any piecewise linear map associated to $\{X_i\}_{i=1}^{62}$.
The invariant set for $S_6$ has Hausdorff dimension equal to
$0.8600\ldots$,
and its Lyapunov spectrum has domain $[\log(5/4),\log(40960000)] = [ 0.2231\ldots, 17.5281\ldots]$.
\end{example}

\begin{example} (The map $S_8$)

\noindent
Let $\{X_i\}_{i=1}^{821}$ be a collection of 821 pairwise disjoint closed sub-intervals of $[0,1]$,
where one interval has length $4/5$, 125 intervals have length $1/2000=(4/5)(1/1600)$,
585 intervals have length $1/3200000=(4/5)(1/1600)^2$,
109 intervals have length $1/5120000000=(4/5)(1/1600)^3$,
and the remaining interval has length $1/8192000000000 =(4/5)(1/1600)^4$, and
let $S_8$ be any piecewise linear map associated to $\{X_i\}_{i=1}^{821}$.
The invariant set for $S_8$ has Hausdorff dimension equal to
$0.865154\ldots$,
and its Lyapunov spectrum has domain $[\log(5/4),\log(8192000000000)] = [ 0.2231\ldots, 29.7341\ldots]$.
\end{example}

\begin{theorem}\label{468}
For $n=4, 6,8$, the Lyapunov spectrum for the map $S_n$ has precisely $n$ points of inflection.
\end{theorem}
\begin{proof}
In view of Proposition \ref{keyequationprop} it suffices to show that there are precisely $n$ solutions 
$s_1< s_2<\ldots < s_n$ to equation (\ref{keyequation_initial}).
With $F(t) = \sum_{i=1}^q |X_i|^{t}$ and $p(t) = \log F(-t)$, the corresponding $n$ points of inflection
of $L$ are (by Proposition \ref{keyequationprop}) given by 
$$\alpha_i = p'(-s_i) = \frac{-F'(s_i)}{F(s_i)}$$ 
for $i=1,\ldots, n$.

Writing 
$$G(s):=\frac{1}{2\log F(s)} \ , \quad H(s):=\frac{F(s)F''(s)}{F'(s)^2} -1\ ,$$ 
and $I:=G-H$, the solutions of 
(\ref{keyequation_initial}) are precisely the zeros of the function $I$.

For the map $S_4$, the corresponding function $F$ is given explicitly by
$$
F(s) = \left(\frac{4}{5}\right)^s + 5\left(\frac{1}{65}\right)^s + \left(\frac{1}{3380}\right)^s\,,
$$
while for $S_6$ it is given by
$$
F(s) = \left(\frac{4}{5}\right)^s + 30\left(\frac{1}{400}\right)^s  + 30\left(\frac{1}{128000}\right)^s + \left(\frac{1}{40960000}\right)^s\,,
$$
and for $S_8$ it is given by
$$
F(s) = \left(\frac{4}{5}\right)^s + 125\left(\frac{1}{2000}\right)^s  + 585\left(\frac{1}{3200000}\right)^s + 109\left(\frac{1}{5120000000}\right)^s+ \left(\frac{1}{8192000000000}\right)^s\,.
$$

In each of the three cases the corresponding function $I$, which we denote respectively by $I_4$, $I_6$, $I_8$,
can be graphed (see Figures \ref{i4}, \ref{i6}, \ref{i8}, together with graphs of the corresponding second derivative of the Lyapunov spectrum), and its zeros located (to arbitrary precision, using e.g.~the bisection method).
For the map $S_4$, four zeros of $I_4$ are located,
at $s_1\approx-0.3975$, 
$s_2 \approx -0.0565$, 
$s_3\approx0.0616$, 
$s_4\approx0.4751$,
with corresponding Lyapunov inflections
$\alpha_1\approx5.9895$,
$\alpha_2 \approx 4.4269$,
$\alpha_3 \approx 3.8988$,
$\alpha_4\approx2.0166$.
For the map $S_6$, six zeros of $I_6$ are located,
at
$s_1\approx -0.5654$,
$s_2\approx-0.4724$,
$s_3\approx-0.1534$,
$s_4\approx0.2269$,
$s_5\approx0.4275$,
$s_6\approx0.6236$,
with corresponding Lyapunov inflections
$\alpha_1\approx 14.2732$,
$\alpha_2\approx 13.3822$,
$\alpha_3\approx10.4393$,
$\alpha_4\approx6.6173$,
$\alpha_5\approx4.7925$,
$\alpha_6\approx2.9345$.
For the map $S_8$, eight zeros of $I_8$ are located, at
$s_1\approx-0.5655$,
$s_2\approx-0.4827$,
$s_3\approx-0.4538$,
$s_4\approx-0.0660$,
$s_5\approx0.0649$,
$s_6\approx0.4150$,
$s_7\approx0.4891$,
$s_8\approx0.6883$,
with corresponding Lyapunov inflections
$\alpha_1\approx24.6075$,
$\alpha_2\approx23.1999$,
$\alpha_3\approx22.7113$,
$\alpha_4\approx15.8995$,
$\alpha_5\approx13.7839$,
$\alpha_6\approx7.8754$,
$\alpha_7\approx6.7335$,
$\alpha_8\approx3.5981$.

\begin{figure}[!h]
\begin{center}
\includegraphics[scale=1.0]{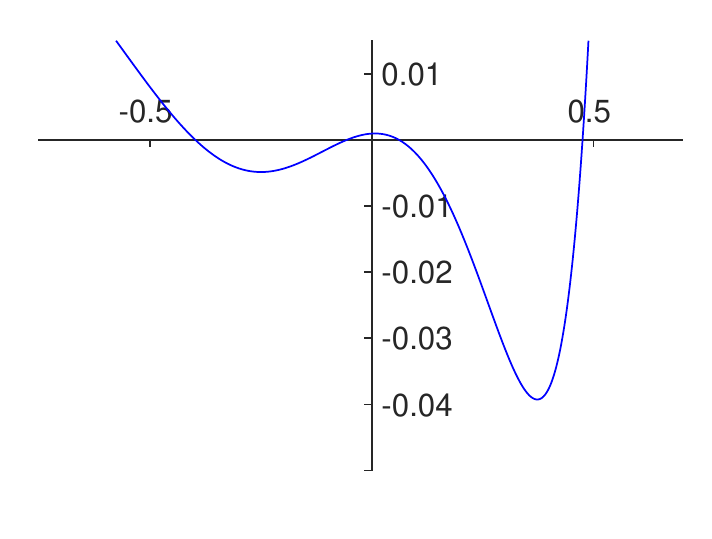}
\includegraphics[scale=1.0]{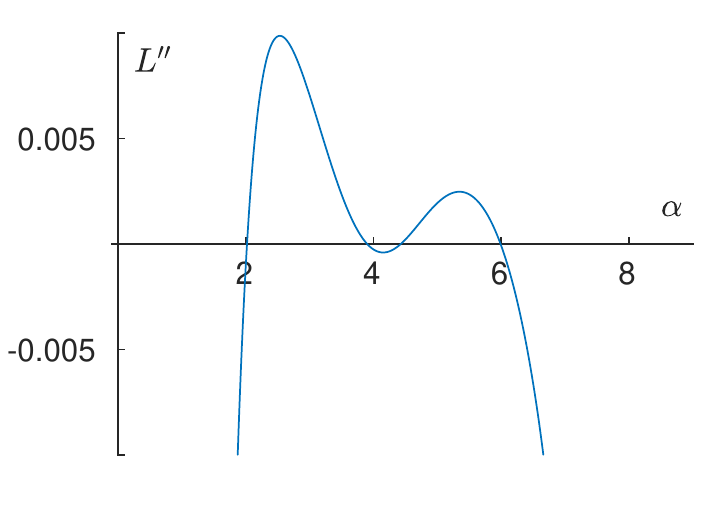}
\caption{Function $I_4$ and the corresponding second derivative $L''$.}\label{i4}
\end{center}
\end{figure}
\begin{figure}[!h]
\begin{center}
\includegraphics[scale=1.0]{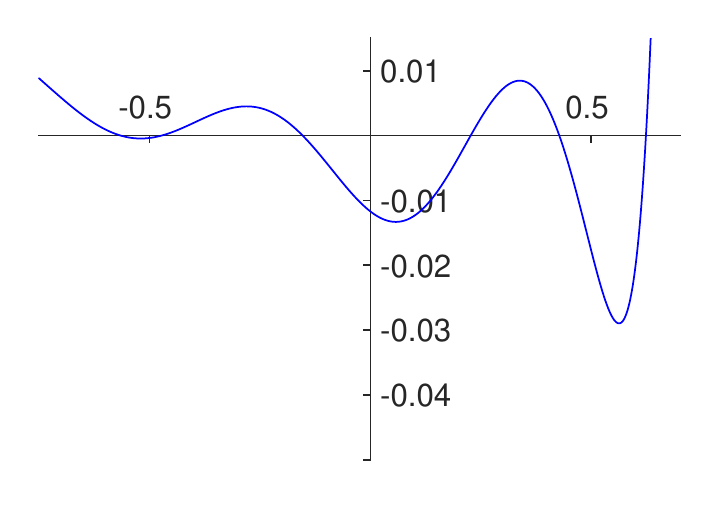}
\includegraphics[scale=1.0]{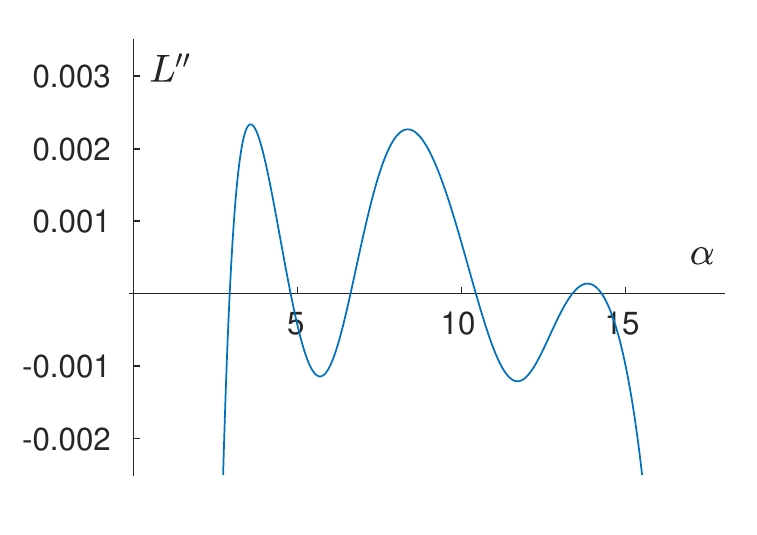}
\caption{Function $I_6$ and the corresponding second derivative $L''$.}\label{i6}
\end{center}
\end{figure}
\begin{figure}[!h]
\begin{center}
\includegraphics[scale=1.0]{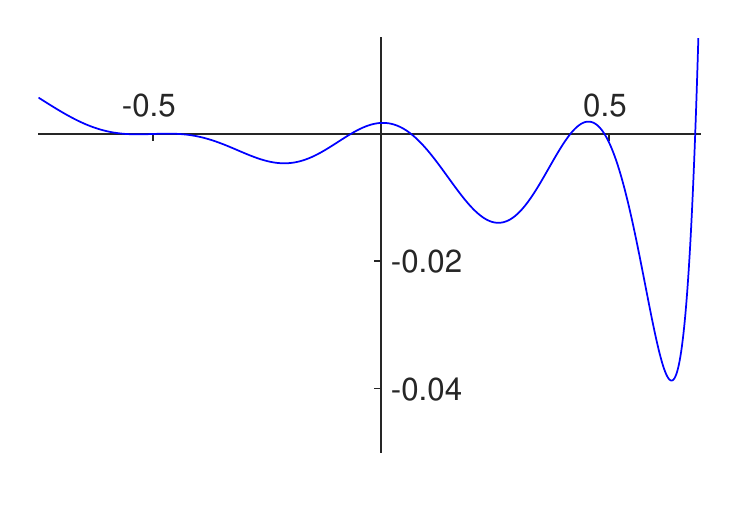} 
\includegraphics[scale=1.0]{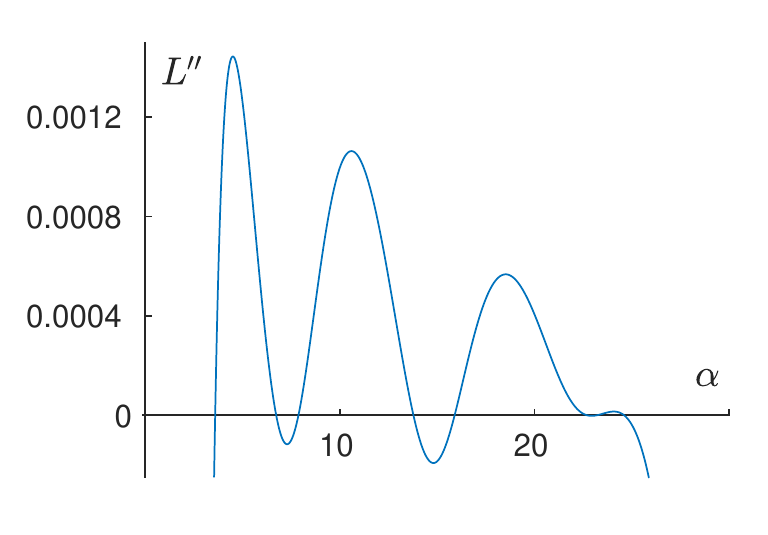}
\caption{Function $I_8$  and the corresponding second derivative $L''$.}\label{i8}
\end{center}
\end{figure}

In each case, the fact that the above are the only zeros of $I=I_n$
can be proved via an intermediate value argument together with
piecewise monotonicity properties.
Specifically, $2n$ points $s_1^-<s_1^+<s_2^-<s_2^+<\ldots<s_n^-<s_n^+$
can be chosen, where $s_n^+<\dim_H(X)$, such that $I(s_i^-)$ and $I(s_i^+)$ have opposite sign,
and $I$ can be shown to be strictly monotone on $(s_i^-,s_i^+)$ (by bounding $I'$ away from zero).
Moreover $I$ can be bounded away from zero on the complementary intervals
$(-\infty, s_1^-)$, $(s_1^+,s_2^-)$, $(s_2^+,s_3^-)$,\ldots, $(s_{n-1}^+,s_n^-)$,
while on $(s_n^+,\infty)$ the function $I$ is strictly positive on $(s_n^+,\dim_H(X))$, has a singularity at
$\dim_H(X)$ (since $F(\dim_H(X))=1$), and is strictly negative on $(\dim_H(X),\infty)$.
\end{proof}

In view of Theorem \ref{twotheorem}, and the 7-branch map $S_4$ with four Lyapunov inflections,
it is natural to wonder if there is a 3-branch piecewise linear map with four Lyapunov inflections,
or failing that whether four Lyapunov inflections occurs for any
piecewise linear map with strictly fewer than 7 branches.
More generally:

\begin{question}\label{minbranches}
For any (even) natural number $n$, what is the smallest number $Q_n$ such that there exists a $Q_n$-branch
piecewise linear map whose Lyapunov spectrum has $n$ points of inflection?
\end{question}

As a consequence of  Theorems \ref{finiteinflectionstheorem}
and \ref{468}, together with \cite{iommikiwi}, we have the following bounds:

\begin{cor}
For all even natural numbers $n$, the smallest number $Q_n$ of branches of a 
piecewise linear map needed to witness $n$ points of inflection in its Lyapunov spectrum satisfies
\begin{enumerate}
\item[\, (i)] $Q_2=2$,
\item[\, (ii)] $Q_4\le 7$,
\item[\, (iii)] $Q_6\le 62$,
\item[\, (iv)] $Q_8\le 821$,
\item[\, (v)] $Q_n\le\sum_{j=6}^{n/2 +26} 2^{j^2} < 2^{(n/2 +26)^2 +1} $ for all $n$.
\end{enumerate}
\end{cor}



\end{document}